\documentclass[12pt, a4paper, reqno]{amsart}
\usepackage{amsmath}
\usepackage{amsfonts}
\usepackage{amssymb}
\usepackage{amsthm}
\usepackage{url}
\usepackage{latexsym}
\usepackage{bbm}
\usepackage{enumerate}
\usepackage{booktabs}
\usepackage[T1]{fontenc}

\newcommand{\dx}{\,dx}
\newcommand{\dmu}{\,d\mu}
\newcommand{\E}{\mathrm{e}}
\newcommand{\R}{\mathbb{R}}
\newcommand{\Z}{\mathbb{Z}}
\newcommand{\Rn}{\mathbb{R}^n}

\newcommand{\BMO}{\mathrm{BMO}}
\newcommand{\Ar}{\Arrowvert}

\DeclareMathOperator{\bmo}{BMO}

\def\mvint_#1{\mathchoice%
          {\mathop{\kern 0.2em\vrule width 0.6em height 0.69678ex depth -0.58065ex
                  \kern -0.8em \intop}\nolimits_{\kern -0.4em#1}}%
          {\mathop{\kern 0.1em\vrule width 0.5em height 0.69678ex depth -0.60387ex
                  \kern -0.6em \intop}\nolimits_{#1}}%
          {\mathop{\kern 0.1em\vrule width 0.5em height 0.69678ex depth -0.60387ex
                  \kern -0.6em \intop}\nolimits_{#1}}%
          {\mathop{\kern 0.1em\vrule width 0.5em height 0.69678ex depth -0.60387ex
                  \kern -0.6em \intop}\nolimits_{#1}}}

\def\sqr#1#2{{\vcenter{\hrule height.#2pt\hbox{\vrule
    width.#2pt height#1pt\kern#1pt\vrule width.#2pt}\hrule height.#2pt}}}
\def\square{\mathchoice\sqr{5.5}4 \sqr{5.0}4 \sqr{4.8}3 \sqr{4.8}3}

\def\XXint#1#2#3{{\setbox0=\hbox{$#1{#2#3}{\int}$}
     \vcenter{\hbox{$#2#3$}}\kern-.5\wd0}}

\theoremstyle{plain}
\newtheorem{theorem}[equation]{Theorem}
\newtheorem{lemma}[equation]{Lemma}

\newtheorem{proposition}[equation]{Proposition}

\numberwithin{equation}{section}

\theoremstyle{definition}
\newtheorem{definition}[equation]{Definition}
\newtheorem{example}[equation]{Example}

\theoremstyle{remark}
\newtheorem{remark}[equation]{Remark}
\title{John-Nirenberg lemmas for a doubling measure}
\author{Daniel Aalto \and Lauri Berkovits \and Outi Elina Maasalo \and Hong Yue}
\date{\today}

\begin{document}

\subjclass[2000]{42B25}

\begin{abstract}
We study, in the context of doubling metric measure spaces, a class of $\bmo$ type 
functions defined by John and Nirenberg. In particular, we present a new version of the 
Calder\'on--Zygmund decomposition in metric spaces and use it to prove the 
corresponding John--Nirenberg 
inequality.
\end{abstract}

\maketitle

\section{Introduction}

Besides the well known class of functions of bounded mean oscillation,
$\bmo$, F. John and L. Nirenberg 
defined another, larger class of functions in their paper~\cite{JoNi61}. We call this  
space John-Nirenberg space with exponent $p$  and write JN$_p$. Whereas the classical John-Nirenberg 
lemma shows that any function of bounded mean oscillation has exponentially decaying distribution 
function, any function in JN$_p$ belongs to weak $L^p$. 

Unlike $\bmo$, the John-Nirenberg space has not been systematically studied. In this paper we 
generalize the definition into doubling metric measure spaces by replacing the cubes in the original 
definition by metric space balls, and, in particular, prove the John-Nirenberg lemma for JN$_p$ in this
 setting. We also study properties of the space, and, for example, show that every $p$-integrable function 
 is in the John-Nirenberg space with the same exponent, and provide an example of a function in the 
 weak $L^p$ that is not a John-Nirenberg function.

In the Euclidean case there are a few proofs of the John-Nirenberg inequality for JN$_p$. The original
 proof in~\cite{JoNi61}, based on an induction argument, can be found with more details in~\cite{Gi03}.
 There is an alternative proof in the real line; see~\cite{To86}. As these proofs are rather difficult 
 to follow, we also present here a new proof in the Euclidean case. The proof is based on iterating
 a suitable good-$\lambda$-inequality. It is interesting that this proof generalizes directly to the 
 setting of doubling metric measure spaces via dyadic sets defined by
 M. Christ; see \cite{AiBeIa05} or \cite{Ch90}
 for the definition.

To prove the John-Nirenberg inequality for JN$_p$ in the metric case we have adapted ideas from 
A. P. Calder\'on's proof of 
the classical 
John-Nirenberg lemma for $\bmo$ in the Euclidean setting in \cite{Ne77}, and from the
 aforementioned proof in \cite{To86}.
To this end, we present a new version 
of the Calder\'on-Zygmund decomposition in metric spaces. The advantage of this version is that we 
are able to iterate it efficiently, which is not trivial in the metric setting. We also get both lower 
and upper bounds for mean values over the decomposition balls. Existence of a doubling measure is the
 only assumption we need to impose on the space. 

Calder\'on's method is remarkably flexible as illustrated by a 
simplified proof of the so-called parabolic John-Nirenberg inequality
by E. Fabes and N. Garofalo; see \cite{FG}.
To further demonstrate this flexibility of Calder\'on's technique and the use of our decomposition lemma 
we also give a new proof of the classical 
John-Nirenberg lemma for $\bmo$ in doubling metric measure spaces. The lemma has previously been 
generalized into doubling metric measure spaces, for example, in \cite{Bu99}, \cite{MaPe02}, 
\cite{MaMaNiOr00}. 

\emph{Acknowledgements.} D.A. and O.E.M. were supported 
by the Finnish Academy of Science and Letters Vilho, Yrj\"o and Kalle V\"ais\"al\"a Foundation; L.B. was supported 
by the Finnish Cultural Foundation, North Ostrobothnia Regional fund;
H.Y. was supported by the Academy of Finland.
The authors would like to thank A. Bj\"orn, J. Kinnunen and X. Tolsa for valuable discussions.

\section{Doubling metric measure spaces}

Let $(X,d,\mu)$ be a metric space endowed with a metric $d$ and a Borel regular measure $\mu$. We assume that an open ball always comes with a center and a radius, i.e.
$$
B=B(x,r)=\{y\in X:d(y,x)<r\}.
$$
We denote with $\lambda B$ the $\lambda$-dilate of $B$, that is a ball with the same center as $B$ but $\lambda$ times its radius.
We assume that $\mu$ is doubling, i.e. all open balls have positive and finite measure whenever $r>0$ and there exists a constant $c_{\mu}\geq 1$, called the doubling constant of $\mu$, so that
$$
\mu (2B)\leq c_{\mu}\mu(B)
$$
for all $B$ in $X$. 

The doubling condition implies a covering theorem, sometimes referred to as the Vitali covering theorem. Indeed, given any collection of balls with uniformly bounded radius, there exists a pairwise disjoint, countable subcollection of balls, whose 5-dilates cover the union of the original collection. This theorem implies Lebesgue's differentiation theorem, which guarantees that any locally integrable function can be approximated at almost every point by integral averages of the function over a contracting sequence of balls.

The Hardy-Littlewood maximal function $Mf$ of a locally integrable function $f$ is defined
for every $x\in X$ by
$$
Mf(x) = \sup_{ B\ni x} \mvint_B|f|\dmu ,
$$
where 
$$
f_B=\mvint_Bf\dmu = \frac{1}{\mu(B)}\int_Bf\dmu,
$$
and the supremum is taken over all balls containing $x$. The Hardy-Littlewood maximal function satisfies 
\begin{equation}\label{strongpp}
\|Mf\|_p\leq c(p,\mu)\|f\|_p
\end{equation}

for every $f\in L^p(X)$ with $1<p\leq \infty$. For the proof of \eqref{strongpp}, Vitali covering theorem and futher information on metric spaces, see, for example, \cite{He01}.


\section{The second John-Nirenberg inequality for a doubling measure}

We begin by recalling the definition of the John-Nirenberg space in the Euclidean case; see \cite{JoNi61}. Let $Q_0$ be a cube in $\Rn$ and $1 \leq p < \infty$. 
An integrable function $f$ defined on $Q_0$ belongs to $JN_p(Q_0)$, the John-Nirenberg 
space with exponent $p$, if there exists $K_f < \infty$ so that
\begin{equation}\label{csum}
\sum_i|Q_i|\left[\mvint_{Q_i}|f -f_{Q_i}|\dx\right]^p \leq K_f ^p
\end{equation}
independent of the family $\{Q_i\}_{i=1}^{\infty}$, where $Q_i$ are subcubes of $Q_0$ 
such that $\bigcup Q_i=Q_0$ and the interiors of $Q_i$ are disjoint.

Observe that the definition given by cubes can be directly generalised in metric spaces.
Indeed, the dyadic structure of the Euclidean cubes can be transferred to a doubling metric measure space using Christ's construction \cite{Ch90}. Then the natural definition is in terms of these dyadic sets. However, the definition of JN$_p$ in a doubling metric measure space is most natural in terms of balls. Balls cannot be organised in a simple dyadic way in nested generations as cubes in $\R^n$ and we have thus chosen to define the space JN$_p$ so that the definition is compatible with the Vitali covering theorem.

\begin{definition}
Let $(X,d,\mu)$ be a metric measure space, $1< p <\infty$ and $B_0\subset X$ a ball.
Let $f$ be a locally integrable function defined on $11B_0$. We say that $f$ belongs to the John-Nirenberg space with exponent $p$, and we
write $f \in JN_p(B_0)$, if there exists $K_f < \infty$ so that
$$
\sum_i \mu(B_i) \left( \mvint_{B_i} |f-f_{B_i}| \dmu \right)^p \leq K_f^p 
$$
whenever $\{B_i\}$ is a countable collection of balls centered at $B_0$ and contained 
in $11B_0$ with the property that the balls $\frac{1}{5}B_i$ are pairwise disjoint. 
We will call the smallest possible constant $K_f$ the JN$_p$ norm of $f$. 
 \end{definition}

\begin{remark}
Observe that JN$_p$ is a generalization of BMO. Indeed, it follows 
directly from the definitions that a function is of bounded mean oscillation 
if and only if the JN$_p$ norm is bounded as $p$ tends to infinity.
\end{remark}

The next result shows that there are plenty of functions in John-Nirenberg spaces.

\begin{proposition}
\label{subset}
Let $1<p<\infty$ and $f\in L^p(11B_0)$. Then $f\in JN_p(B_0)$.
\end{proposition}
\begin{proof}
Let $B_i$ be a family of balls that is admissible in the definition of $JN_p(B_0)$. 
Write $B_i'=\frac{1}{5}B_i$ for the disjoint balls.
We know that for every ball $B_i$ it holds that
$$
\mvint_{B_i'}Mf\dmu
\geq
\inf_{x\in B_i'}Mf(x) 
\geq 
\mvint_{B_i}|f|\dmu.
$$
Hence,
$$
\sum_{i}\mu(B_i)\left(\mvint_{B_i}|f-f_{B_i}|\dmu\right)^p
\leq
2^pc_{\mu}^3\sum_{i}\mu(B_i')\left(\mvint_{B_i'}Mf\dmu\right)^p.
$$
Now by H\"older's inequality
$$
\mu(B_i')\left(\mvint_{B_i'}Mf\dmu\right)^p
\leq
\int_{B_i'}(Mf)^p\dmu,
$$
and by the disjointness of the balls $B_i'$ and by the boundedness of the maximal operator we have
$$
\sum_i\int_{B_i'}(Mf)^p\dmu
\leq
\int_{11B_0}(Mf)^p\dmu
\leq
c\int_{11B_0}|f|^p\dmu,
$$
which is finite by the assumption. This completes the proof.
\end{proof}

Notice, that in $\R^n$ Proposition~\ref{subset} follows from the definition simply by using the H\"older inequality. 


The John-Nirenberg inequality for JN$_p(Q_0)$ shows that it is contained in weak $L^p(Q_0)$. The following 
 one-dimensional example shows that the inclusion is strict.
 
\begin{example}
\label{notlp}
Consider the function $f(x)=x^{-\frac{1}{p}}$ on $Q_0=(0,2)$ with $p>1$. 
It is clear that the function belongs to weak $L^p(Q_0)$. 
Let us partition the interval $Q_0$ as
$Q_j=(2^{-j},2^{1-j})$, where $j=0,1, \ldots$,  
to see that (\ref{csum}) fails.
A simple change of variable $x=2^{-j}y$ shows that
$f_{Q_j}=2^{j/p}f_{Q_0}$. Similarly, we set
$I=|f-f_{Q_0}|_{Q_0}$ and conclude that $|f-f_{Q_j}|_{Q_j}=2^{j/p}I$.
Hence, the sum in (\ref{csum}) diverges.
\end{example}

The following theorem is our main result.

\begin{theorem}
\label{mainresult}
If $f \in JN_p(B_0)$, then
\begin{equation}
\mu(\{x \in B_0: |f(x)-f_{B_0}|>\lambda\}) \leq C\left(\frac{K_f}{\lambda}\right)^p,
\end{equation}
where $C$ only depends on $p$ and the doubling constant.
\end{theorem}

To prove the theorem we need two lemmas. The first one is a Calderon-Zygmund decomposition lemma 
and the second one is a good-$\lambda$ type inequality. 
The key idea behind the proof of 
Theorem~\ref{mainresult} stems from the method used in \cite{Ne77}.

\begin{lemma}
\label{CZ}
Let $f$ be a non-negative locally integrable funtion on $X$. Fix a ball $B_0=B(x_0,R)$ and assume that
\begin{equation*}
\lambda_0 \geq \frac{1}{\mu(B_0)} \int_{11B_0} f \dmu.
\end{equation*}
Then there exists a countable, possibly finite, family of disjoint balls 
$\{B_i\}_i$ centered in $B_0$ and
satisfying $5B_i \subset 11B_0$ so that
\begin{enumerate}[i)] 
\item $f(x)\leq \lambda_0$ for $\mu$-a.e. $x \in B_0\setminus \bigcup_i 5B_i$,
\item $\lambda_0 < \mvint_{B_i} f\leq c_\mu^3 \lambda_0$,
\item $c_\mu^{-3}{\lambda_0} < 
\mvint_{5B_i} f \leq \lambda_0$.
\end{enumerate}
The balls satisfying the above conditions are called Calder\'on-Zygmund balls at level $\lambda_0$.
Moreover, if $\lambda_0 \leq \lambda_1 \leq \ldots \leq \lambda_N$, 
then the 
Calder\'on-Zygmund balls corresponding 
to different levels $\lambda_n$ may be chosen in such a way that each
$B_i(\lambda_{n+1})$ is contained in some $5B_j(\lambda_n)$.
\end{lemma}

\begin{proof}
Define a maximal function
\begin{equation*}
M_{B_0}f(x) = \sup_{\substack{ B \ni x \\ B\subset B_0 }} \mvint_{B} f \dmu,
\end{equation*}
where the supremum is taken over all balls containing $x$ and included in $B_0$. Write 
\[
 E_\lambda = \{ x \in B_0: M_{B_0}f(x) > \lambda \}.
\]

Let us first consider $\lambda_N$ to show how the balls are chosen.
By the definition of $M_{B_0}f$, for every $x \in E_{\lambda_N}$ 
there exists a ball $B_x$ with $x \in B_x \subset B_0$ and
\begin{equation}\label{Eq1}
\lambda_0 \leq \ldots \leq \lambda_N < \mvint_{B_x} f \dmu.
\end{equation}
We now take a look at the balls $5^kB_x,$ where $k \in \Z_+$.
Note that if a ball $B$ satisfies $B_0 \subset B \subset 11B_0$, then by
the choice of $\lambda_N$, we have
$$
\mvint_{B} f \dmu \leq \frac{1}{\mu(B_0)} 
\int_{11B_0} f \dmu \leq
\lambda_0 \leq \lambda_N.
$$
If $B_x$ has radius $r$, take $k$ such that $5^{k-1}r\leq2R<5^kr$. Then 
$B_0 \subset 5^kB_x \subset 11B_0$ and the average of $f$ over $5^kB_x$ 
is at most $\lambda_N$.
Consequently, there exists a smallest $n=n_x \geq1$ such that
\begin{equation}\label{Eq2}
\mvint_{5^nB_x} f \dmu \leq \lambda_N.
\end{equation}
Then
\begin{equation}\label{Eq3}
\lambda_N < \mvint_{5^jB_x} f \dmu
\end{equation}
for all $j=0,1,...n-1$.

Consider the balls $5^{n_x-1}B_x$. They form a covering of $E_{\lambda_N}$ and 
by the Vitali covering theorem we
may pick a countable subfamily of pairwise disjoint 
balls $B_i= 5^{n_{x_i}-1}B_{x_i}$ with
$$
E_{\lambda_N} \subset \bigcup_{i=1}^\infty 5B_i.
$$
The balls $B_i$ have the required properties.
Indeed, by \eqref{Eq2} and \eqref{Eq3}, we have
\begin{equation}\label{Eq5}
\lambda_N <  \mvint_{5^{n-1}B_x} f \dmu 
\leq c_\mu^3 \mvint_{5^nB_x} f
 \dmu \leq c_\mu^3 
\lambda_N,
\end{equation}thus proving ii).
Since $5B_i=5^nB_{x_i}$, the first inequality in iii) has already been proved in 
\eqref{Eq5}, while the second inequality is just \eqref{Eq2}.
It remains to prove i). We have $$B_0\setminus \bigcup_{i=1}^\infty 5B_i 
\subset B_0 \setminus E_{\lambda_N}.$$
This implies that
$
M_{B_0}f(x) \leq \lambda_N \text{ for $\mu$-a.e. } x \in B_0\setminus \bigcup_i 5B_i, 
$
from which we get i) by Lebesgue's differentiation theorem.

We have now constructed the desired decomposition at level $\lambda_N$ and
 turn to $\lambda_{N-1}$. 
Since $E_{\lambda_N} \subset E_{\lambda_{N-1}}$ 
for every $x \in E_{\lambda_N}$ we may start from exactly the same ball 
$B_x$ satisfying \eqref{Eq1} as before.
For every $x \in E_{\lambda_{N-1}}\setminus E_{\lambda_N}$ we take a ball
$B_x$ with $x \in B_x \subset B_0$ and
\begin{equation}
\lambda_0 \leq \ldots \leq \lambda_{N-1} < 
\mvint_{B_x} f \dmu.
\end{equation}
Now for each ball $B_x$ choose the smallest $m=m_x \geq 1$ satisfying
\begin{equation}\label{}
\mvint_{5^mB_x} f \dmu \leq \lambda_{N-1}.
\end{equation}
Notice that if $B_x$ is a ball corresponding to an $x \in E_{\lambda_N}$, 
then $n\leq m$ (here $n$ is from \eqref{Eq2}).
Then apply Vitali's theorem to the balls $5^{m-1}B_x$ to obtain a family of
balls satisfying properties i)-iii) with $\lambda_0$ replaced by 
$\lambda_{N-1}$.

Now let $B_i(\lambda_N)$ by any of the Calder\'on-Zygmund balls corresponding
to $\lambda_N$. Then $B_i(\lambda_N)=5^{n-1}B_{x_i}$ for some $x_i 
\in E_{\lambda_N}$ and $B_i(\lambda_N) \subset 5^{m-1}B_{x_i}$ (because $n\leq m$).
The ball $5^{m-1}B_{x_i}$ is not necessarily a Calder\'on-Zygmund ball 
corresponding to the level $\lambda_{N-1}$, but it is one of the balls 
in the collection from which 
the Calder\'on-Zygmund balls were extracted. Vitali's theorem shows
that $5^{m-1}B_{x_i}$ is contained in a 5-dilate of
some of them, say, $B_{j}(\lambda_{N-1})$. Then $B_i(\lambda_N) \subset 5B_{j}(\lambda_{N-1})$.

We continue this procedure. Next, we consider $E_{\lambda_{N-2}}$. For
$x \in E_{\lambda_N}$ we take the same ball $B_x$ which we used 
in the first step. 
For $x \in E_{\lambda_{N-1}}\setminus E_{\lambda_N}$ 
 we use the same ball 
$B_x$ which we used 
in the second step. For every $x \in 
E_{\lambda_{N-2}}\setminus E_{\lambda_{N-1}}$ we take a ball $B_x$ with 
$x \in B_x \subset B_0$ and
\begin{equation}
\lambda_0 \leq \ldots \leq \lambda_{N-2} < 
\mvint_{B_x} f \dmu
\end{equation}
and proceed as previously.
\end{proof}

\begin{lemma}
\label{toiterate}
Assume $f \in JN_p(B_0)$ 
and
$$
\lambda\geq \frac{1}{\mu(B_0)} \int_{11B_0} |f-f_{B_0}| \dmu.
$$
Consider Calder\'on-Zygmund balls 
$\{B_i(\lambda)\}_i$ and $\{B_j(2\lambda)\}_j$ for the function $|f-f_{B_0}|$ at levels $\lambda$ and $2\lambda$, respectively.
Suppose that each $B_i(2\lambda)$ is contained in some $5B_j(\lambda)$. 
Then we have
\begin{equation}\label{Eq4}
\sum_j \mu(B_j(2\lambda))\leq \frac{c_\mu^{3/q}K_f}{\lambda}
\left( \sum_i \mu(B_i(\lambda)) \right)^{1/q},
\end{equation}
where $q$ is the conjugate exponent of $p$, that is $1/p+1/q=1$.
\end{lemma}

\begin{proof}
We may assume $K_f=1$ and $f_{B_0}=0$. We partition the 
family $\{B_j(2\lambda)\}_j$ as follows. First collect those which are contained
 in $5B_1(\lambda)$. From the remaining balls we 
collect those which are contained in
$5B_2(\lambda)$ and continue similarly. In other words,
$$
\{B_j(2\lambda)\}_j= \bigcup_i \{B_j(2\lambda)\}_{j \in J_i} 
$$
where
\begin{align*}
&J_1=\{j: B_j(2\lambda) \subset 5B_1(\lambda)\},\\
&J_2=\{j: B_j(2\lambda) \subset 5B_2(\lambda), j \notin J_1\},\\
&J_3=\{j: B_j(2\lambda) \subset 5B_3(\lambda), j \notin J_1 \cup J_2\},\\
&\,\,\,\,\,\,\,\,\,\,\vdots
\end{align*}
We have
\begin{equation}\label{}
2\lambda \sum_j \mu(B_j(2\lambda)) \leq \sum_j \int_{B_j(2\lambda)} |f| \dmu=
\sum_i \sum_{j \in J_i} \int_{B_j(2\lambda)} |f| \dmu,
\end{equation}
where
\begin{align*}
\sum_{j \in J_i} \int_{B_j(2\lambda)} |f| \dmu & \leq \sum_{j \in J_i} \int_{B_j(2\lambda)} 
||f|+\lambda-|f_{5B_i(\lambda)}|| \dmu \\
& \leq \sum_{j \in J_i} \int_{B_j(2\lambda)} |f-f_{5B_i(\lambda)}| \dmu 
+\sum_{j \in J_i} \int_{B_j(2\lambda)} \lambda \dmu \\
& \leq \int_{5B_i(\lambda)} |f-f_{5B_i(\lambda)}| \dmu +\lambda\sum_{j \in J_i} \mu(B_j(2\lambda)).
\end{align*}
Now we sum over $i$ to obtain 
$$
2\lambda \sum_j \mu(B_j(2\lambda)) \leq \sum_i \int_{5B_i(\lambda)} |f-f_{5B_i(\lambda)}| 
\dmu +\lambda \sum_j \mu(B_j(2\lambda)).
$$
By H\"older's inequality and the normalization $K_f=1$ we get
\begin{align*}
\sum_i \int_{5B_i(\lambda)} &|f-f_{5B_i(\lambda)}| \dmu\\ =&
\sum_i \mu(5B_i(\lambda))^{1/q}\mu(5B_i(\lambda))^{-1/q}  \int_{5B_i(\lambda)} 
|f-f_{5B_i(\lambda)}| \dmu \\
\leq &\left( \sum_i \mu(5B_i(\lambda)) \right)^{1/q}  \\
& \quad\cdot
\left( \sum_i \mu(5B_i(\lambda))^{-p/q} \left(\int_{5B_i(\lambda)} |f-f_{5B_i(\lambda)}| \dmu\right)^p \right)^{1/p}\\
\leq &c_\mu^{3/q}\left( \sum_i \mu(B_i(\lambda)) \right)^{1/q},
\end{align*}
whence
$$
2\lambda \sum_j \mu(B_j(2\lambda)) \leq c_\mu^{3/q}
\left( \sum_i \mu(B_i(\lambda)) \right)^{1/q} +\lambda \sum_j \mu(B_j(2\lambda)).
$$
This finishes the proof.
\end{proof}

{\bf\emph{Proof of Theorem~\ref{mainresult}}.}
We  wish to iterate the estimate \eqref{Eq4}. We still assume $K_f=1$ and $f_{B_0}=0$, whence
$$
\mu(B_0)\left( \mvint_{B_0} |f| \dmu\right)^p \leq 1
$$
and
$$
\mu(11B_0)\left( \mvint_{11B_0} |f-f_{11B_0}| \dmu\right)^p \leq 1.
$$
Therefore,
\begin{equation*}
\begin{split}
\frac{1}{\mu(B_0)}\int_{11B_0}& |f| \dmu  \leq
c_\mu^4 \mvint_{11B_0} |f-f_{11B_0}| \dmu
+c_\mu^4 \mvint_{B_0} |f_{11B_0}| \dmu\\
&\leq \frac{c_\mu^4}{\mu(11B_0)^{1/p}}+c_\mu^4 \mvint_{B_0} |f-f_{11B_0}| \dmu
+c_\mu^4 \mvint_{B_0} |f| \dmu\\
&\leq \frac{c_\mu^4}{\mu(B_0)^{1/p}}+c_\mu^8 \mvint_{11B_0} |f-f_{11B_0}| \dmu
+\frac{c_\mu^4}{\mu(B_0)^{1/p}}\\
&\leq \frac{C_1}{\mu(B_0)^{1/p}},
\end{split}
\end{equation*}
where $C_1=3c_\mu^8$. We choose
$$
\lambda_0=\frac{C_1}{\mu(B_0)^{1/p}}.
$$
Now let $\lambda > \lambda_0$ and take $N \in \Z_+$ such that
\begin{equation}\label{rajat}
2^N \lambda_0 < \lambda \leq 2^{N+1}\lambda_0.
\end{equation}
Then apply the decomposition lemma at levels $\lambda_0 <2\lambda_0
<2^2\lambda<\ldots <2^N\lambda$ to obtain $N+1$ families of Calder\'on-Zygmund balls. 
Observe that for $n=0,1,\ldots,N-1$ each $B_i(2^{n+1}\lambda)$ is 
contained in some
$5B_j(2^n\lambda)$.

First notice that
\begin{multline*}
\mu(\{x \in B_0: |f(x)|>\lambda\}) \leq \mu(\{x \in B_0: |f(x)|>2^N\lambda_0\}) \\
 \leq \sum_j \mu(5B_j(2^N\lambda_0)) \leq c_\mu^3\sum_j \mu(B_j(2^N\lambda_0)).
\end{multline*}
Then use \eqref{Eq4} and  the fact that
$$
1+q^{-1}+\ldots+q^{-(N-1)} = p-pq^{-N}
$$
to estimate
\begin{align*}
\sum_j \mu(B_j&(2^N\lambda_0))
\leq \frac{c_\mu^{3/q}}{2^{N-1}\lambda_0}
\left(\frac{c_\mu^{3/q}}{2^{N-2}\lambda_0}\right)^{1/q} 
\left(\frac{c_\mu^{3/q}}{2^{N-3}\lambda_0}\right)^{1/q^2}\cdot \ldots\\ 
&\,\,\,\,\,\,\,\cdot \left( \frac{c_\mu^{3/q}}{2^{0}\lambda_0}\right)^
{1/q^{N-1}}\cdot\left( \frac{1}{\lambda_0} \int_{11B_0} |f| \dmu\right)^{q^{-N}}\\
&=\frac{c_\mu^{3q^{-1}+3q^{-2}+\ldots3q^{-N}}}{g(N)}
\cdot\left(\frac{1}{\lambda_0}\right)^{p-pq^{-N}}
\cdot\left( \frac{1}{\lambda_0} \int_{11B_0} |f| \dmu\right)^{q^{-N}}.\\
\end{align*}
Here $g(1)=1$ and for $N\geq2$,
$$
\frac{1}{g(N)}
=\frac{2^{q^{-1}+2q^{-2}+\ldots+(N-1)q^{-(N-1)}}}{2^{(N-1)(p-pq^{-N})}}
$$
We have the estimate 
$$
\frac{c_\mu^{3q^{-1}+3q^{-2}+\ldots3q^{-N}}}{g(N)}
 \leq \frac{C}{2^{(N-1)p}},
$$
where the constant $C$ only depends on $p$ and the doubling constant.

Moreover, the choice of $\lambda_0$ gives
\begin{equation*}
 \begin{split}
\left(\frac{1}{\lambda_0}\right)^{-pq^{-N}}
\cdot\bigg( \frac{1}{\lambda_0} &\int_{11B_0} |f| \dmu\bigg)^{q^{-N}} \\ 
&\leq \left(\frac{C_1^p}{\mu(B_0)}\right)^{q^{-N}} \cdot
\mu(B_0)^{q^{-N}}=C_1^{pq^{-N}}\leq C_1^{pq}.
\end{split}
\end{equation*}
Now combine the previous estimates and use \eqref{rajat} to get
\begin{align*}
\mu(\{x \in B_0: |f(x)|>\lambda\})  \leq \frac{C}{2^{(N-1)p}} 
\left(\frac{1}{\lambda_0}\right)^p 
=\frac{C}{(2^{N-1}\lambda_0)^p}
 \leq \frac{C}{\lambda^{p}}.
\end{align*}
Here $C$ is a constant depending only on $p$ and on the doubling constant. For $0<\lambda<\lambda_0$ we use the trivial estimate
$$
\mu(\{x \in B_0: |f(x)| > \lambda \}) \leq \mu(B_0) =
\frac{C_1^p}{\lambda_0^p}\leq \frac{C_1^p}{\lambda^p}.
$$
\hfill{$\square$}

\section{Euclidean case}
In this section we give a new proof for the second John-Nirenberg ineaquality in $\R^n$. See Lemma 3 in \cite{JoNi61}.

\begin{theorem}[John-Nirenberg inequality II]\label{the1} If $f$ is a function 
satisfies \eqref{csum}, then $f-f_{Q_0}$ 
is in weak $L^p(Q_0)$,
 i.e., there exists $C>0$ depending only on $n$ and $p$ such that
\begin{equation}\label{JNlemma3}
|\{x\in Q_0:|f(x)-f_{Q_0}|>{\lambda}\}|\leq C\left(\frac{K_f}{\lambda}\right)^p
\end{equation}
for all $\lambda>0$. 
\end{theorem}

Let ${Q}$ be a cube in $\Rn$ with sides parallel to the coordinate axes, and, 
denote by $|S|$ the Lebesgue measure of a set $S$. The dyadic maximal function of $f$ is defined as
\begin{equation}\label{maxf}
{M}^df(x)=\sup_{Q\ni x}\mvint_{Q}|f(y)|dy,
\end{equation}
where the supremum is taken over all dyadic cubes $Q$ containing $x$. Moreover, 
for $\lambda>0$  
 we define ${E}_Q({\lambda})=\{x\in Q: {M}^df(x)>\lambda\}$.

We recall a decomposition lemma; see \cite{St93}, Chapter IV, Section 3.1.

\begin{lemma}\label{lem1}
Let $Q_0$ be a cube and let
$f \in L^1(Q_0)$. 
Suppose that
$$
\mvint_{Q_0}|f(x)|\dx\leq\lambda.
$$
Then we have $E_{Q_0}(\lambda)=\bigcup_{k=1}^{\infty} Q_k$, where $\{Q_k\}$ is a
collection of cubes whose interiors are disjoint, such that 
\begin{enumerate}[i)]
\item $|f(x)|\leq\lambda$ \ for  a.e.\ $x\in Q_0\setminus\bigcup_{k=1}^{\infty} Q_k$,
\item  $\lambda < \mvint_Q|f(x)|\dx\leq 2^n\lambda$, for all
$Q$ in the collection $\{Q_k\}$,
\item  {\em $|E_{Q_0}(\lambda)|\leq \frac{1}{\lambda}\int_{E_{Q_0}(\lambda)}|f(x)| \dx$}. 
\end{enumerate}
\end{lemma}

The following good-$\lambda$-inequality is the core of our proof.
\begin{lemma}\label{pro1} For a function $f \in JN_p(Q_0)$ and 
a number $0<b<2^{-n}$ we have 
\begin{multline}\label{equ1}
|\{x\in Q_0:M^d(f-f_{Q_0})(x)>{\lambda}\}|
\\
\leq 
\frac{aK_f}{\lambda}|\{x\in Q_0:M^d(f-f_{Q_0})(x)>{b\lambda}\}|^{1/q}
\end{multline}
for all $\lambda\geq\frac{1}{b}\mvint_{Q_0}|f(x)-f_{Q_0}|\dx$, where $a=\frac{1}{(1-2^nb)}$.
\end{lemma}
\begin{proof}
\noindent
Without loss of generality, we assume that $f_{Q_0}=0$, then \eqref{equ1} becomes
\begin{equation}\label{equ2}
|E_{Q_0}(\lambda)|\leq \frac{aK_f}{\lambda}\cdot  |E_{Q_0}(b\lambda)|^{1/q}.\end{equation}

First, we apply Lemma~\ref{lem1} to $|f(x)|$ on $Q_0$ with $\lambda$ replaced by $b\lambda$ to get a collection of countable 
disjoint dyadic cubes $\{Q_k\}_{k\geq 1}$, such that $E_{Q_0}(b\lambda)=\bigcup_{k=1}^{\infty}Q_k$. It follows that $E_{Q_0}(\lambda)=\bigcup_{k=1}^{\infty}E_{Q_k}(\lambda)$ since $E_{Q_0}(\lambda)\subset E_{Q_0}(b\lambda)$. 

Moreover, let $x\in Q_k$ be such that $M^df(x)>\lambda$. Then there exists a 
dyadic cube $Q$ containing $x$ with
\begin{equation}
 \label{bound}
\mvint_{Q}f\dx >\lambda.
\end{equation}
Since $Q_k$ is the maximal dyadic cube such that the first inequality in $ii)$ holds for $b\lambda$, $Q\subset Q_k$ and it follows from \eqref{bound} that $M^d(f\chi_{Q_k})(x)>\lambda$. Moreover, $M^d[(f-f_{Q_k})\chi_{Q_k}](x)>(1-2^nb)\lambda$ by the second inequality in $ii)$.

Then fix a $k$, if $\mvint_{Q_k}|(f-f_{Q_k})|\dx\leq (1-2^nb)\lambda$, we apply Lemma~\ref{lem1} to $|(f-f_{Q_k})\chi_{Q_k}|$ on $Q_k$ with $\lambda$ replaced by $(1-2^nb)\lambda$.  By $iii)$ we have
\begin{equation}
\label{Qk}
 \begin{split}
|E_{Q_k}(\lambda)| &\leq |\{x\in Q_k:M^d[(f-f_{Q_k})\chi_{Q_k}](x)>(1-2^nb)\lambda\}| \\
&\leq \frac{1}{(1-2^nb)\lambda}\int_{Q_k}|f-f_{Q_k}| \dx \\
&=\frac{|Q_k|^{1/q}}{(1-2^nb)\lambda}\left({|Q_k|^{1/p-1}}\int_{Q_k}|f-f_{Q_k}| \dx\right).
\end{split}
\end{equation}
Otherwise $|Q_k|<\frac{1}{(1-2^nb)\lambda}\int_{Q_k}|f-f_{Q_k}| \dx$, and \eqref{Qk} holds as well.

By adding these inequalities for all $k$, we get, by the H\"older inequality
\begin{equation*}
\begin{split}
|E_{Q_0}&(\lambda)|\leq\sum_{k}\frac{|Q_k|^{1/q}}{(1-2^nb)\lambda}\left({|Q_k|^{1/p-1}}\int_{Q_k}|f-f_{Q_k}| \dx\right)\\
&\leq\frac{1}{(1-2^nb)\lambda}\left(\sum |Q_k|\right)^{1/q}\left\{\sum {|Q_k|^{1-p}}\left[\int_{Q_k}|f-f_{Q_k}| \dx\right]^p\right\}^{1/p}\\
&\leq\frac{1}{(1-2^nb)\lambda}|E_{Q_0}(b\lambda)|^{1/q}K_{f}
\end{split}
\end{equation*}
since $Q_k$ are disjoint. 
\end{proof}

We are now ready to prove the John-Nirenberg lemma.

\noindent
{\bf\emph{Proof of Theorem \ref{the1}.}}
Without loss of generality we may assume $f_{Q_0}=0$. Let $b=2^{-(n+1)}$ 
and define
$$
\eta =\frac{K_f}{b\cdot |Q_0|^{1/p}}.
$$
Let
$$
\lambda\geq\frac{1}{b}\mvint_{Q_0}|f(x)|\dx
$$
and let $j$ be the smallest integer satisfying $b^{-j}\eta<\lambda$. We iterate the estimate \eqref{equ2} $j$ times to get
\begin{equation*}
 \begin{split}
|E_\lambda(Q_0)|&\leq \left|E_{Q_0}\left(b^{-j}\eta\right)\right|\\
&\leq\left(\frac{aK_f}{b^{-j}\eta}\right)\left(\frac{aK_f}{b^{-j+1}\eta}\right)^{1/q}\cdots \left(\frac{aK_f}{b^{-1}\eta}\right)^{1/q^{j-1}}\left|E_{Q_0}(\eta)\right|^{1/q^j}\\
&\leq\left(\frac{aK_f}{b\lambda}\right)\left(\frac{aK_f}{b^{2}\lambda}\right)^{1/q}\cdots \left(\frac{aK_f}{b^{j}\lambda}\right)^{1/q^{j-1}}\left[\frac{1}{\eta}{\int_{Q_0}|f|\dx}\right]^{1/q^j},
\end{split}
\end{equation*}
where the third inequality comes from the weak type inequality $iii)$
and the definition of $j$.

Observe that
$$
1+\frac{2}{q}+\cdots+\frac{j}{q^{j-1}} \leq p^2.
$$
By the definition of JN$_p$ and $\eta$ we have
$$
\frac{1}{\eta}{\int_{Q_0}|f|\dx} \leq b|Q_0|.
$$
Hence,
\begin{align*}
|E_{Q_0}(\lambda)|
&\leq
\left(\frac{aK_f}{\lambda}\right)^{p(1-q^{-j})}b^{-p^2}(b|Q_0|)^{1/q^j}\\
&=2^{p(1-q^{-j})}2^{(n+1)(p^2-1/q^j)}\left(\frac{K_f}{\lambda}\right)^{p}\left|\frac{\lambda |Q_0|^{1/p}}{K_f}\right|^{p/q^j}.
\end{align*}
By the definition of $\eta$ and $j$ we have that 
$$
\frac{\lambda|Q_0|^{1/p}}{K_f}\leq b^{-j+2}=2^{(n+1)(j-2)}.
$$
Since
$$
(j-2)q^{-j}\leq q^{-3}p^2,
$$
we can now conclude
$$
|E_{Q_0}(\lambda)|
\leq 
2^{p+(n+1)(p^2+(\frac{p}{q})^3)}\left(\frac{K_f}{\lambda}\right)^{p}.
$$
This proves the theorem for large values of $\lambda$.

For $\lambda\leq\frac{K_f}{b |Q_0|^{1/p}}$, we have
$$
|E_{Q_0}(\lambda)|
\leq
|Q_0|
\leq
2^{(n+1)p}\left(\frac{K_f}{\lambda}\right)^p
$$
as desired.

Observe that this proof can be generalized to the metric setting via Christ's dyadic sets and by a Calder\'on-Zygmund decomposition lemma by Aimar \& al.; see  Theorems 2.6 and 3.1 in \cite{AiBeIa05}.


\section{John-Nirenberg inequality for a doubling measure}

In this section we give a new proof of the John-Nirenberg lemma in a doubling metric measure space. The result is by no means sharper or more general than the results in the literature. Nevertheless, we wish that the current proof will further increase the understanding of the phenomenon.

We recall that a locally integrable function $f\colon X \to \R$ is in $\bmo(X)$ if there exists a constant $c$ such that
\begin{equation}
\label{bmo}
\mvint_{B}|f-f_B|\dmu\leq c
\end{equation}
for all balls $B$ in $X$. The space is equipped with the seminorm
\begin{equation*}
\|f\|_{\sharp}=\sup_{B\subset X}\mvint_{B}|f-f_B|\dmu.
\end{equation*}
 If we define an equivalence relation
\[f\sim g \quad \textrm{if and only if}\quad f-g=constant,
\]
then $\bmo(X)/_{\sim}$ is a normed space. As is common, we continue denoting the space $\bmo(X)$ and speak of functions instead of equivalence classes.

\begin{theorem}
Let $f \in \BMO (X)$. Then we have
$$
\mu (\{x\in B: |f-f_B|>\lambda \})
\leq
c_1 \mu(B)\, e^{-c_2\frac{\lambda}{\|f\|_{\sharp}}}
$$
for all balls $B\subset X$ and $\lambda >0$ with
with $c_1,c_2$ not depending on $f$ and $\lambda$.
\end{theorem}

\begin{proof}
Take $f \in \BMO(X)$. We may assume that $\Ar f \Ar_\sharp=1$. 
We first notice that
\begin{align*}
\frac{1}{\mu(B)} \int_{11B} |f-f_B| \dmu &\leq c_\mu^4 \mvint_{11B}
|f-f_{11B}| \dmu+c_\mu^4|f_B-f_{11B}|\\
&\leq c_\mu^4+c_\mu^4\mvint_B |f-f_{11B}| \dmu\\
&\leq 2c_\mu^8.
\end{align*}
Thus, the expression on the left-hand side above is bounded uniformly
in $B$. 
Now fix a ball $B_0$ and assume $f_{B_0}=0$. If $\{B_j\}_j$
is the Calder\'on-Zygmund decomposition at level $\lambda\geq 2c_\mu^8$,
given by Lemma \ref{CZ}, then
\begin{enumerate}[i)] 
\item $|f(x)|\leq \lambda$ for $\mu$-a.e. $x \in B_0\setminus 
\bigcup_j 5B_j$,
\item $\lambda < \mvint_{B_j} |f|\leq c_\mu^3 \lambda$,
\item $c_\mu^{-3}{\lambda} < 
\mvint_{5B_j} |f| \leq \lambda$.
\end{enumerate}
We have by i) that
\begin{equation}\label{CM2}
\mu(\{ x \in B_0: |f(x)|>\lambda \}) \leq \sum_j \mu(5B_j) 
\leq c_\mu^3\sum_j \mu(B_j).
\end{equation}
Analogously to Calder\'on's proof in \cite{Ne77}, we wish to study the size of 
$\sum_j \mu(B_j)$.
Apply the decomposition lemma at levels $\lambda >\gamma 
\geq 2c_\mu^8$.
Denote the corresponding Calder\'on-Zygmund balls by 
$\{B_j(\lambda)\}_j$ and $\{B_k(\gamma)\}_k$, which we choose
in a similar way as in the proof of Lemma~\ref{toiterate}.
We write 
$\{B_j(\lambda)\}_j$ as a disjoint union
$$
\{B_j(\lambda)\}_j= \bigcup_k \{B_j(\lambda)\}_{j \in J_k} 
$$
where
$J_k$'s are defined as in the proof of Lemma~\ref{toiterate}, but $2\lambda$ replaced by $\lambda$, and, $\lambda$ by $\gamma$.
By ii), we may now write
\begin{equation}\label{CM1}
\lambda \sum_j \mu(B_j(\lambda)) \leq \sum_j \int_{B_j(\lambda)} |f| \dmu=
\sum_k \sum_{j \in J_k} \int_{B_j(\lambda)} |f| \dmu.
\end{equation}
Moreover, we have
\begin{align*}
\sum_{j \in J_k} \int_{B_j} |f| \dx & \leq \sum_{j \in J_k} \int_{B_j} 
||f|+\gamma-|f_{5B_k(\gamma)}|| \dx \\
& \leq \sum_{j \in J_k} \int_{B_j} |f-f_{5B_k(\gamma)}| \dx +\sum_{j \in J_k} \int_{B_j} \gamma \dx \\
& \leq \int_{5B_k(\gamma)} |f-f_{5B_k(\gamma)}| \dx +\gamma\sum_{j \in J_k} \mu(B_j) \\
& \leq \mu(5B_k(\gamma))+\gamma\sum_{j \in J_k} \mu(B_j) \\
& \leq c_\mu^3\mu(B_k(\gamma))+\gamma\sum_{j \in J_k} \mu(B_j).
\end{align*}
Now sum over $k$ and use \eqref{CM1} to obtain 
$$
\lambda \sum_j \mu(B_j(\lambda)) \leq c_\mu^3
\sum_k \mu(B_k(\gamma))+\gamma \sum_{j} \mu(B_j(\lambda)).
$$
Thus, we see that 
\begin{equation}\label{}
(\lambda-\gamma)\sum_j \mu(B_j(\lambda)) \leq c_\mu^3
\sum_k \mu(B_k(\gamma)).
\end{equation}
whenever $\lambda\geq \gamma \geq 2c_\mu^8$. Now
set $a=2c_\mu^8>2c_\mu^3$ and replace $\lambda$ and $\gamma$ 
respectively by $\lambda+a$ and $\lambda$.
We have shown that if $\lambda \geq a$ and the 
Calder\'on-Zygmund balls corresponding 
to $\lambda$ 
and $\lambda+a$ are chosen in such a way that each ball
$B_j(\lambda+a)$ is contained in some $5B_k(\lambda)$, then
$$
\sum_j \mu(B_j(\lambda+a)) \leq \frac{1}{2} \sum_k \mu(B_k(\lambda)).
$$ 
Now let $\lambda \geq a$ and take $N \in \Z_+$ such that
$
Na \leq \lambda < (N+1)a. 
$ 
Then apply the decomposition lemma at each level $a <2a<\ldots <Na$.
From the above estimate and \eqref{CM2} we get
\begin{equation*}
 \begin{split}
 \mu(\{ x \in B_0: &|f(x)|>\lambda \})  \leq \mu(\{ x \in B_0: |f(x)|>Na \})\\
& \leq c_\mu^3 \sum_j\mu(B_j(Na)) \leq c_\mu^32^{-N+1}\sum_j \mu(B_j(a))\\
&\leq  c_\mu^32^{-N+1}\mu(11B_0)\leq c_\mu^7\E^{(2-\lambda/a)\log 2}\mu(B_0)\\
&=4c_\mu^7\E^{-(\lambda \log 2)/a}\mu(B_0).
\end{split}
\end{equation*}
For $0<\lambda<a$ we have
\begin{equation*}
 \begin{split}
\mu(\{ x \in B_0: |f(x)|>&\lambda \}) \leq \mu(B_0) \\
&\leq 4c_\mu^7\E^{-\log 2}\mu(B_0)\leq 4c_\mu^7\E^{-(\lambda \log 2)/a}\mu(B_0).
\end{split}
\end{equation*}
Hence the John-Nirenberg inequality holds with $c_1=4c_\mu^7$
and $c_2=(\log 2)/a$.
\end{proof}


\vspace{0.5cm}
\noindent
\small{\textsc{D.A.},}
\small{\textsc{Department of Mathematics},}
\small{\textsc{FI-20014 University of Turku},}
\small{\textsc{Finland}}\\
\footnotesize{\texttt{daniel.aalto@iki.fi}}

\vspace{0.5cm}
\noindent
\small{\textsc{L. B.},}
\small{\textsc{Department of Mathematics},}
\small{\textsc{FI-90014 University of Oulu},}
\small{\textsc{Finland}}\\
\footnotesize{\texttt{lauri.berkovits@oulu.fi}}

\vspace{0.5cm}
\noindent
\small{\textsc{O. E. M.},}
\small{\textsc{Mathematisches Institut},}
\small{\textsc{Universit\"at Bern},}
\small{\textsc{3012 Bern, Switzerland}}\\
\footnotesize{\texttt{outi.elina.maasalo@tkk.fi}}

\vspace{0.5cm}
\noindent
\small{\textsc{H. Y.},}
\small{\textsc{Department of Mathematics and Informatics},}
\small{\textsc{Trine University},}
\small{\textsc{Angola, IN, USA 46703}}\\
\footnotesize{\texttt{yueh@trine.edu}}

\end{document}